\documentclass{amsart}

\newtheorem{theorem}{Theorem}[section]
\newtheorem{lemma}[theorem]{Lemma}

\theoremstyle{definition}
\newtheorem{definition}[theorem]{Definition}
\newtheorem{example}[theorem]{Example}
\newtheorem{xca}[theorem]{Exercise}

\theoremstyle{remark}
\newtheorem{remark}[theorem]{Remark}

\numberwithin{equation}{section}

\newcommand{\abs}[1]{\lvert#1\rvert}

\newcommand{\blankbox}[2]{%
  \parbox{\columnwidth}{\centering
    \setlength{\fboxsep}{0pt}%
    \fbox{\raisebox{0pt}[#2]{\hspace{#1}}}%
  }%
}

\usepackage{enumerate}

\usepackage{xcolor} 
\definecolor{tssteelblue}{RGB}{70,130,180}
\definecolor{tsorange}{RGB}{255,138,88}
\definecolor{ocre}{RGB}{52,177,201} 

\newcommand{\tssteelblue}[1]{{\color{tssteelblue}#1}}

\definecolor{tssteelblue}{rgb}{1,1,1}

\definecolor{rosybrown}{RGB}{188,143,143} 

\definecolor{coral}{RGB}{240,128,128} 

\definecolor{salmon}{RGB}{255,160,122} 

\definecolor{lightpink}{RGB}{255,182,193} 

\definecolor{pink}{RGB}{255,192,203} 

\definecolor{palevioletred}{RGB}{219,112,147} 

\definecolor{deeppink}{RGB}{255,20,147} 
\newcommand{\deeppink}[1]{{\color{deeppink}#1}}

\definecolor{violet}{RGB}{238,130,238} 

\definecolor{mediumorchid}{RGB}{186,85,211} 

\definecolor{darkmagenta}{RGB}{139,0,139} 
\newcommand{\darkmagenta}[1]{{\color{darkmagenta}#1}}

\definecolor{indianred}{RGB}{205,92,92} 

\definecolor{maroon}{RGB}{128,0,0} 

\definecolor{darkcyan}{RGB}{0,139,139} 

\definecolor{gold}{RGB}{255,215,0} 

\definecolor{orange}{RGB}{255,165,0} 

\definecolor{mediumblue}{RGB}{0,0,205} 

\definecolor{mediumseagreen}{RGB}{60,179,113} 

\definecolor{darkolivegreen}{RGB}{85,107,47} 

\definecolor{lightseagreen}{RGB}{32,178,170} 

\definecolor{deeppink}{rgb}{0,0,0}

\definecolor{darkmagenta}{rgb}{0,0,0}

\begin{document}

\title{On generalized Li-Yau inequalities}


\author{Li-Chang Hung}
\address{Department of Mathematics, Soochow University, Taipei, Taiwan}
\curraddr{Department of Mathematics, Soochow University, Taipei, Taiwan}
\email{lichang.hung@gmail.com }
\thanks{
}




\date{
}



\begin{abstract}

We generalize the Li-Yau inequality for second derivatives and we also establish Li-Yau type inequality for fourth derivatives. Our derivation relies on the representation formula for the heat equation.



\end{abstract}

\maketitle


.

\section{Introduction}
The Li-Yau inequality asserts that if $u=u(x,t)$ is a positive solution to the heat equation, then the logarithm of $u(x,t)$ forms a supersolution to Laplace's equation, i.e.
\begin{equation}\label{eqn: Li-Yau inequality original}
\Delta\log u\ge -\frac{n}{2t},   
\end{equation}
where $n$ is the dimension of the manifold and $t>0$. The inequality \eqref{eqn: Li-Yau inequality original} is due to Li and Yau \cite{li1986parabolic}. Their derivation relies on an idea related to the parabolic maximum principle. 

Our goal here is to give a direct proof of a generalization of inequality \eqref{eqn: Li-Yau inequality original} without using the parabolic maximum principle. Our proof relies on the representation formula for the heat equation $u_t=\Delta u$ \cite{evans2022partial}
\begin{equation}\label{eqn: representation formula for the heat equation Li-Yau} 
u(x,t)=\displaystyle\frac{1}{(4\pi t)^{n/2}}\int_{\mathbb{R}^n} g(y)e^{-\frac{|x-y|^{2}}{4t}}\,dy \quad (x\in\mathbb{R}^n,\; t>0),
\end{equation}   
where the initial condition $g\in C(\mathbb{R}^n)\bigcap L^{\infty}(\mathbb{R}^n)$ is assumed to be nonnegative and does not vanish completely. To generalize inequality \eqref{eqn: Li-Yau inequality original}, we need the following two lemmas.

\begin{lemma}[Jensen inequality]\label{lem: Holder's inequality => Jensen inequality}
Let $p\ge1$ and suppose that $\int|g|=1$. Then 
\begin{equation}
\left(
\int|fg|
\right)^p\le
\int
|f|^{p}|g|.
\end{equation}  

\end{lemma}

\begin{proof}

\begin{align}
\nonumber
\int|fg|
&=\int|f||g|^{\frac{1}{p}+\frac{1}{q}}
&& \text{\tssteelblue{let $q$ satisfy $\frac{1}{p}+\frac{1}{q}=1$
}} 
\\[1ex] \notag
&=
\int|f||g|^{\frac{1}{p}}
\cdot
|g|^{\frac{1}{q}}
&& \text{\tssteelblue{$|g|^{\frac{1}{p}+\frac{1}{q}}=|g|^{\frac{1}{p}}
\cdot
|g|^{\frac{1}{q}}$
}} 
\\[1ex] \notag
&\le
\left(
\int
\left(
|f||g|^{\frac{1}{p}}
\right)^{p}
\right)^{\frac{1}{p}}
\left(
\int
\left(
|g|^{\frac{1}{q}}
\right)^{q}
\right)^{\frac{1}{q}}
&& \text{\tssteelblue{H\"older's inequality: $\int|fg|
\le
\left(\int |f|^p\right)^{1/p}
\left(\int |g|^q\right)^{1/q}$
}}
\\[1ex] \notag
&\le
\left(
\int
|f|^{p}|g|
\right)^{\frac{1}{p}}
\left(
\int
|g|
\right)^{\frac{1}{q}}
&& \text{\tssteelblue{simplify
}}
\\[1ex] \notag
&\le
\left(
\int
|f|^{p}|g|
\right)^{\frac{1}{p}}
&& \text{\tssteelblue{$\int|g|=1$
}}
\\[1ex] \notag
\implies \left(
\int|fg|
\right)^p&\le
\int
|f|^{p}|g|
&& \text{\tssteelblue{$p$-th power
}}
\end{align}

\end{proof}

Assume \eqref{eqn: representation formula for the heat equation Li-Yau} holds. Then we can find the derivatives of $u$ in the following Lemma~\ref{lem: derivatives of representation formula solution of heat equation}.

\begin{lemma}\label{lem: derivatives of representation formula solution of heat equation}
\ \
\begin{enumerate}[(a)]
  \item 
  
\begin{equation}  
\frac{u_{x_i}(x,t)}{u(x,t)}=\int_{\mathbb{R}^n}-\frac{x_{i}-y_{i}}{2t}
\frac{\frac{1}{(4\pi t)^{n/2}}e^{-\frac{|x-y|^{2}}{4t}}g(y)}{u(x,t)}
\,dy.
\end{equation}  
  
  \item 
  
\begin{equation}  
\frac{u_{x_i x_i}(x,t)}{u(x,t)}=\int_{\mathbb{R}^n}
\left(
-\frac{1}{2t}
+
\frac{(x_{i}-y_{i})^2}{4t^2}
\right)
\frac{\frac{1}{(4\pi t)^{n/2}}e^{-\frac{|x-y|^{2}}{4t}}g(y)}{u(x,t)}
\,dy.
\end{equation}    
  
  \item For $i\neq j$,
\begin{equation}  
\frac{u_{x_i x_j}(x,t)}{u(x,t)}=\int_{\mathbb{R}^n}
\frac{(x_{i}-y_{i})(x_{j}-y_{j})}{4t^2}
\frac{\frac{1}{(4\pi t)^{n/2}}e^{-\frac{|x-y|^{2}}{4t}}g(y)}{u(x,t)}
\,dy.
\end{equation}    
  \item 
  
\begin{equation}  
\frac{u_{x_i x_i x_i}(x,t)}{u(x,t)}=\int_{\mathbb{R}^n}
\frac{(x_{i}-y_{i})(6t-(x_{i}-y_{i})^2)}{8t^3}
\frac{\frac{1}{(4\pi t)^{n/2}}e^{-\frac{|x-y|^{2}}{4t}}g(y)}{u(x,t)}
\,dy.
\end{equation} 
  \item 
  
\begin{equation}  
\frac{u_{x_i x_i x_i x_i}(x,t)}{u(x,t)}=\int_{\mathbb{R}^n}
\frac{12 t^2-12 t (x_{i}-y_{i})^2+(x_{i}-y_{i})^4}{16 t^4}
\frac{\frac{1}{(4\pi t)^{n/2}}e^{-\frac{|x-y|^{2}}{4t}}g(y)}{u(x,t)}
\,dy.
\end{equation} 

  \item For $i\neq j$,
  
\begin{equation}  
\frac{u_{x_i x_i x_j x_j}(x,t)}{u(x,t)}=\int_{\mathbb{R}^n}
\frac{((x_{i}-y_{i})^2-2t)((x_{j}-y_{j})^2-2t)}{16t^4}
\frac{\frac{1}{(4\pi t)^{n/2}}e^{-\frac{|x-y|^{2}}{4t}}g(y)}{u(x,t)}
\,dy.    
\end{equation}

\end{enumerate}


\end{lemma}

\begin{proof}
Straightforward calculations lead to (a), (b), and (c). 
\begin{enumerate}
  \item [(d)]
  
\begin{align}
\nonumber
&\frac{u_{x_i x_i x_i}(x,t)}{u(x,t)}
:=\frac{(u_{x_i x_i}(x,t))_{x_i}}{u(x,t)}
&& \text{\tssteelblue{Lemma~\ref{lem: derivatives of representation formula solution of heat equation} (b)
}} 
\\[1ex] \notag
&=
\int_{\mathbb{R}^n}
\left[
\frac{x_{i}-y_{i}}{2t^2}
-
\frac{x_{i}-y_{i}}{2t}
\left(
-\frac{1}{2t}
+
\frac{(x_{i}-y_{i})^2}{4t^2}
\right)
\right]
\frac{\frac{1}{(4\pi t)^{n/2}}e^{-\frac{|x-y|^{2}}{4t}}g(y)}{u(x,t)}
\,dy
&& \text{\tssteelblue{
}} 
\\[1ex] \notag
&=
\int_{\mathbb{R}^n}
\frac{6t(x_{i}-y_{i})-(x_{i}-y_{i})^3}{8t^3}
\frac{\frac{1}{(4\pi t)^{n/2}}e^{-\frac{|x-y|^{2}}{4t}}g(y)}{u(x,t)}
\,dy
&& \text{\tssteelblue{
}} 
\\[1ex] \notag
&=
\int_{\mathbb{R}^n}
\frac{(x_{i}-y_{i})(6t-(x_{i}-y_{i})^2)}{8t^3}
\frac{\frac{1}{(4\pi t)^{n/2}}e^{-\frac{|x-y|^{2}}{4t}}g(y)}{u(x,t)}
\,dy
&& \text{\tssteelblue{
}} 
\end{align}  
  
  \item [(e)]
  
\begin{align}
\nonumber
&\frac{u_{x_i x_i x_i x_i}(x,t)}{u(x,t)}
:=\frac{(u_{x_i x_i x_i}(x,t))_{x_i}}{u(x,t)}
&& \text{\tssteelblue{Lemma~\ref{lem: derivatives of representation formula solution of heat equation} (d)
}} 
\\[1ex] \notag
&=
\int_{\mathbb{R}^n}
\left[
\frac{6t-3(x_{i}-y_{i})^2}
{8t^3}
-
\left(\frac{x_{i}-y_{i}}{2t}\right)
\left(
\frac{(x_{i}-y_{i})(6t-(x_{i}-y_{i})^2)}{8t^3}
\right)
\right]
\frac{\frac{1}{(4\pi t)^{n/2}}e^{-\frac{|x-y|^{2}}{4t}}g(y)}{u(x,t)}
\,dy
&& \text{\tssteelblue{
}} 
\\[1ex] \notag
&=
\int_{\mathbb{R}^n}
\left(
\frac{12 t^2-12 t (x_{i}-y_{i})^2+(x_{i}-y_{i})^4}{16 t^4}
\right)
\frac{\frac{1}{(4\pi t)^{n/2}}e^{-\frac{|x-y|^{2}}{4t}}g(y)}{u(x,t)}
\,dy
&& \text{\tssteelblue{
}} 
\\[1ex] \notag
&=
\int_{\mathbb{R}^n}
\left(
\frac{-12 t (x_{i}-y_{i})^2+(x_{i}-y_{i})^4}{16 t^4}
\right)
\frac{\frac{1}{(4\pi t)^{n/2}}e^{-\frac{|x-y|^{2}}{4t}}g(y)}{u(x,t)}
\,dy
+\frac{3}{4}\frac{1}{t^2}
&& \text{\tssteelblue{
}}
\end{align}  
  
  \item [(f)]
  
\begin{align}
\nonumber
&\frac{u_{x_i x_i x_j}(x,t)}{u(x,t)}
:=\frac{(u_{x_i x_j}(x,t))_{x_i}}{u(x,t)}
&& \text{\tssteelblue{Lemma~\ref{lem: derivatives of representation formula solution of heat equation} (c)
}} 
\\[1ex] \notag
&=
\int_{\mathbb{R}^n}
\left(
\frac{x_{j}-y_{j}}{4t^2}
-
\frac{x_{i}-y_{i}}{2t}
\frac{(x_{i}-y_{i})(x_{j}-y_{j})}{4t^2}
\right)
\frac{\frac{1}{(4\pi t)^{n/2}}e^{-\frac{|x-y|^{2}}{4t}}g(y)}{u(x,t)}
\,dy
&& \text{\tssteelblue{
}} 
\\[1ex] \notag
&=
\int_{\mathbb{R}^n}
\frac{(x_{j}-y_{j})(2t-(x_{i}-y_{i})^2)}{8t^3}
\frac{\frac{1}{(4\pi t)^{n/2}}e^{-\frac{|x-y|^{2}}{4t}}g(y)}{u(x,t)}
\,dy
&& \text{\tssteelblue{collect $x_{j}-y_{j}$
}} 
\end{align}  
  
\begin{align}
\nonumber
&\frac{u_{x_i x_i x_j x_j}(x,t)}{u(x,t)}
:=\frac{(u_{x_i x_i x_j}(x,t))_{x_j}}{u(x,t)}
&& \text{\tssteelblue{
}} 
\\[1ex] \notag
&=
\int_{\mathbb{R}^n}
\left(
\frac{2t-(x_{i}-y_{i})^2}{8t^3}
-
\frac{x_{j}-y_{j}}{2t}
\frac{(x_{j}-y_{j})(2t-(x_{i}-y_{i})^2)}{8t^3}
\right)
\frac{\frac{1}{(4\pi t)^{n/2}}e^{-\frac{|x-y|^{2}}{4t}}g(y)}{u(x,t)}
\,dy
&& \text{\tssteelblue{
}} 
\\[1ex] \notag
&=
\int_{\mathbb{R}^n}
\frac{4t^2-2t(x_{i}-y_{i})^2-2t(x_{j}-y_{j})^2+(x_{i}-y_{i})^2(x_{j}-y_{j})^2}{16t^4}
\frac{\frac{1}{(4\pi t)^{n/2}}e^{-\frac{|x-y|^{2}}{4t}}g(y)}{u(x,t)}
\,dy
&& \text{\tssteelblue{factored
}} 
\\[1ex] \notag
&=
\int_{\mathbb{R}^n}
\frac{((x_{i}-y_{i})^2-2t)((x_{j}-y_{j})^2-2t)}{16t^4}
\frac{\frac{1}{(4\pi t)^{n/2}}e^{-\frac{|x-y|^{2}}{4t}}g(y)}{u(x,t)}
\,dy
&& \text{\tssteelblue{
}} 
\\[1ex] \notag
&\le
\int_{\mathbb{R}^n}
\frac{(x_{i}-y_{i})^2(x_{j}-y_{j})^2}{16t^4}
\frac{\frac{1}{(4\pi t)^{n/2}}e^{-\frac{|x-y|^{2}}{4t}}g(y)}{u(x,t)}
\,dy
+
\frac{1}{4t^2}
&& \text{\tssteelblue{$\int_{\mathbb{R}^n}\frac{\frac{1}{(4\pi t)^{n/2}}e^{-\frac{|x-y|^{2}}{4t}}g(y)}{u(x,t)}\,dy=1$
}}
\end{align}  
  
\end{enumerate}

\end{proof}

\section{Proof of the generalized Li-Yau inequality}

\begin{theorem}[\deeppink{Generalized Li-Yau inequality}]\label{thm: Li-Yau Type Inequality for second derivatives n dim}
Let $u=u(x,t)$ be given by \eqref{eqn: representation formula for the heat equation Li-Yau}. Then

\begin{equation}
\frac{\Delta u}{u}
-
\alpha
\sum_{i,j=1,i\neq j}^{n}
\frac{u_{x_i x_j}}{u}
-
\beta
\sum_{i,j=1,i\neq j}^{n} 
\frac{u_{x_i}u_{x_j}}{u^2}
-
\gamma
\frac{|\nabla u|^2}{u^2}
\ge-\frac{n}{2t},
\end{equation}

where $\alpha$, $\beta$, and $\gamma$ are nonnegative constants satisfying
\begin{equation}
(n-1)
\left(
\alpha+\beta
\right)
+\gamma
\le1.
\end{equation}

\end{theorem}

\begin{proof}
\ \
\begin{enumerate}
  
  \item 
  Estimate $\left|\frac{\nabla u(x,t)}{u(x,t)}\right|^2$

\begin{align}
\nonumber
&\left|\frac{\nabla u(x,t)}{u(x,t)}\right|^2
:=
\sum_{i=1}^{n} \left|\frac{u_{x_i}(x,t)}{u(x,t)}\right|^2
&& \text{\tssteelblue{definition of $\nabla$
}} 
\\[1ex] \notag
&
=
\sum_{i=1}^{n}
\left|
\int_{\mathbb{R}^n}-\frac{x_{i}-y_{i}}{2t}
\frac{\frac{1}{(4\pi t)^{n/2}}e^{-\frac{|x-y|^{2}}{4t}}g(y)}{u(x,t)}
\,dy
\right|^2
&& \text{\tssteelblue{Lemma~\ref{lem: derivatives of representation formula solution of heat equation} (a)
}} 
\\[1ex] \notag
&\le
\sum_{i=1}^{n}
\int_{\mathbb{R}^n}
\left(\frac{x_{i}-y_{i}}{2t}\right)^2
\frac{\frac{1}{(4\pi t)^{n/2}}e^{-\frac{|x-y|^{2}}{4t}}g(y)}{u(x,t)}
\,dy
&& \text{\tssteelblue{$\int_{\mathbb{R}^n}\frac{\frac{1}{(4\pi t)^{n/2}}e^{-\frac{|x-y|^{2}}{4t}}g(y)}{u(x,t)}\,dy=1$ $\implies$ Lemma\ref{lem: Holder's inequality => Jensen inequality}
}}
\\[1ex] \notag
&=
\int_{\mathbb{R}^n}
\sum_{i=1}^{n}
\left(\frac{x_{i}-y_{i}}{2t}\right)^2
\frac{\frac{1}{(4\pi t)^{n/2}}e^{-\frac{|x-y|^{2}}{4t}}g(y)}{u(x,t)}
\,dy
&& \text{\tssteelblue{$\int_{\mathbb{R}^n}\sum_{i=1}^{n}=\sum_{i=1}^{n}\int_{\mathbb{R}^n}$; $|x-y|^2:=\sum_{i=1}^{n} (x_{i}-y_{i})^2$
}} 
\\[1ex] \label{eqn: heat equation find (nabla u/u)^2 2nd derivative L-Y ineq}
&=
\int_{\mathbb{R}^n}
\frac{|x-y|^2}{4t^2}
\frac{\frac{1}{(4\pi t)^{n/2}}e^{-\frac{|x-y|^{2}}{4t}}g(y)}{u(x,t)}
\,dy
&& \text{\tssteelblue{
}}  
\end{align}

  \item 
  Calculate $\frac{\Delta u(x,t)}{u(x,t)}$
  
\begin{align}
\nonumber
&\frac{\Delta u(x,t)}{u(x,t)}
:=
\sum_{i=1}^{n} \frac{u_{x_i x_i}(x,t)}{u(x,t)}
&& \text{\tssteelblue{definition of $\Delta$
}} 
\\[1ex] \notag
&
=
\sum_{i=1}^{n}
\int_{\mathbb{R}^n}
\left(
-\frac{1}{2t}
+
\frac{(x_{i}-y_{i})^2}{4t^2}
\right)
\frac{\frac{1}{(4\pi t)^{n/2}}e^{-\frac{|x-y|^{2}}{4t}}g(y)}{u(x,t)}
\,dy
&& \text{\tssteelblue{Lemma~\ref{lem: derivatives of representation formula solution of heat equation} (b)
}} 
\\[1ex] \notag
&=
\int_{\mathbb{R}^n}
\sum_{i=1}^{n}
\left(
-\frac{1}{2t}
+
\frac{(x_{i}-y_{i})^2}{4t^2}
\right)
\frac{\frac{1}{(4\pi t)^{n/2}}e^{-\frac{|x-y|^{2}}{4t}}g(y)}{u(x,t)}
\,dy
&& \text{\tssteelblue{$\int_{\mathbb{R}^n}\sum_{i=1}^{n}=\sum_{i=1}^{n}\int_{\mathbb{R}^n}$
}} 
\\[1ex] \label{eqn: heat equation find (Delta u/u) 2nd derivative L-Y ineq}
&=
\int_{\mathbb{R}^n}
\left(
-\frac{n}{2t}
+
\frac{|x-y|^2}{4t^2}
\right)
\frac{\frac{1}{(4\pi t)^{n/2}}e^{-\frac{|x-y|^{2}}{4t}}g(y)}{u(x,t)}
\,dy
&& \text{\tssteelblue{$|x-y|^2:=\sum_{i=1}^{n} (x_{i}-y_{i})^2$
}} 
\end{align}

\item  
Estimate $\sum_{i,j=1,i\neq j}^{n} 
\frac{u_{x_i}(x,t)}{u(x,t)}\cdot\frac{u_{x_j}(x,t)}{u(x,t)}$

\begin{align}
\nonumber
&\sum_{i,j=1,i\neq j}^{n} 
\frac{u_{x_i}(x,t)}{u(x,t)}\cdot\frac{u_{x_j}(x,t)}{u(x,t)}
\le
\sum_{i,j=1,i\neq j}^{n} 
\frac{1}{2}
\left(
\left(
\frac{u_{x_i}(x,t)}{u(x,t)}
\right)^2
+
\left(
\frac{u_{x_j}(x,t)}{u(x,t)}
\right)^2
\right)
&& \text{\tssteelblue{$ab\le\frac{1}{2}(a^2+b^2)$
}} 
\\[1ex] \notag
&
=
\sum_{i,j=1,i\neq j}^{n}
\frac{1}{2}
\left[
\left(
\int_{\mathbb{R}^n}-\frac{x_{i}-y_{i}}{2t}
\frac{\frac{1}{(4\pi t)^{n/2}}e^{-\frac{|x-y|^{2}}{4t}}g(y)}{u(x,t)}
\,dy
\right)^2
+
\left(
\int_{\mathbb{R}^n}-\frac{x_{j}-y_{j}}{2t}
\frac{\frac{1}{(4\pi t)^{n/2}}e^{-\frac{|x-y|^{2}}{4t}}g(y)}{u(x,t)}
\,dy
\right)^2
\right]
&& \text{\tssteelblue{Lemma~\ref{lem: derivatives of representation formula solution of heat equation} (a)
}} 
\\[1ex] \notag
&\le
\sum_{i,j=1,i\neq j}^{n}
\frac{1}{2}
\left[
\int_{\mathbb{R}^n}
\left(\frac{x_{i}-y_{i}}{2t}\right)^2
\frac{\frac{1}{(4\pi t)^{n/2}}e^{-\frac{|x-y|^{2}}{4t}}g(y)}{u(x,t)}
\,dy
+
\int_{\mathbb{R}^n}
\left(\frac{x_{j}-y_{j}}{2t}\right)^2
\frac{\frac{1}{(4\pi t)^{n/2}}e^{-\frac{|x-y|^{2}}{4t}}g(y)}{u(x,t)}
\,dy
\right]
&& \text{\tssteelblue{Lemma~\ref{lem: Holder's inequality => Jensen inequality} 
}}
\\[1ex] \notag
&=
\int_{\mathbb{R}^n}
\sum_{i,j=1,i\neq j}^{n}
\frac{1}{2}
\frac{(x_{i}-y_{i})^2+(x_{j}-y_{j})^2}{4t^2}
\frac{\frac{1}{(4\pi t)^{n/2}}e^{-\frac{|x-y|^{2}}{4t}}g(y)}{u(x,t)}
\,dy
&& \text{\tssteelblue{$\int_{\mathbb{R}^n}\sum=\sum\int_{\mathbb{R}^n}$
}} 
\\[1ex] \notag
&=
\int_{\mathbb{R}^n}
\frac{(n-1)|x-y|^2}{4t^2}
\frac{\frac{1}{(4\pi t)^{n/2}}e^{-\frac{|x-y|^{2}}{4t}}g(y)}{u(x,t)}
\,dy,
&& \text{\tssteelblue{use \eqref{eqn: sum (xi-yi)^2+(xi-yi)^2=2(n-1)|x-y|^2}
}}  
\end{align} 
where we have used the observation
\begin{equation}\label{eqn: sum (xi-yi)^2+(xi-yi)^2=2(n-1)|x-y|^2}
\sum_{i,j=1,i\neq j}^{n}
(x_{i}-y_{i})^2+(x_{j}-y_{j})^2
=2(n-1)|x-y|^2.
\end{equation}

 \item  
 Estimate $\sum_{i,j=1,i\neq j}^{n}\frac{u_{x_i x_j}(x,t)}{u(x,t)}$

\begin{align}
\nonumber
&\sum_{i,j=1,i\neq j}^{n}
\frac{u_{x_i x_j}(x,t)}{u(x,t)}
=
\sum_{i,j=1,i\neq j}^{n}
\int_{\mathbb{R}^n}
\frac{(x_{i}-y_{i})(x_{j}-y_{j})}{4t^2}
\frac{\frac{1}{(4\pi t)^{n/2}}e^{-\frac{|x-y|^{2}}{4t}}g(y)}{u(x,t)}
\,dy
&& \text{\tssteelblue{Lemma~\ref{lem: derivatives of representation formula solution of heat equation} (c)
}} 
\\[1ex] \notag
&\le
\sum_{i,j=1,i\neq j}^{n}
\int_{\mathbb{R}^n}
\frac{1}{2}
\frac{(x_{i}-y_{i})^2+(x_{j}-y_{j})^2}{4t^2}
\frac{\frac{1}{(4\pi t)^{n/2}}e^{-\frac{|x-y|^{2}}{4t}}g(y)}{u(x,t)}
\,dy
&& \text{\tssteelblue{$ab\le\frac{1}{2}(a^2+b^2)$
}} 
\\[1ex] \notag
&=
\int_{\mathbb{R}^n}
\sum_{i,j=1,i\neq j}^{n}
\frac{1}{2}
\frac{(x_{i}-y_{i})^2+(x_{j}-y_{j})^2}{4t^2}
\frac{\frac{1}{(4\pi t)^{n/2}}e^{-\frac{|x-y|^{2}}{4t}}g(y)}{u(x,t)}
\,dy
&& \text{\tssteelblue{$\int_{\mathbb{R}^n}\sum=\sum\int_{\mathbb{R}^n}$
}} 
\\[1ex] \notag
&=
\int_{\mathbb{R}^n}
\frac{(n-1)|x-y|^2}{4t^2}
\frac{\frac{1}{(4\pi t)^{n/2}}e^{-\frac{|x-y|^{2}}{4t}}g(y)}{u(x,t)}
\,dy
&& \text{\tssteelblue{use \eqref{eqn: sum (xi-yi)^2+(xi-yi)^2=2(n-1)|x-y|^2}
}}  
\end{align}

  \item 
  Establish $\frac{\Delta u}{u}
-
\alpha
\sum_{i,j=1,i\neq j}^{n}
\frac{u_{x_i x_j}(x,t)}{u(x,t)}
-
\beta
\sum_{i,j=1,i\neq j}^{n} 
\frac{u_{x_i}(x,t)}{u(x,t)}\cdot\frac{u_{x_j}(x,t)}{u(x,t)}
-
\gamma
\frac{|\nabla u|^2}{u^2}
\ge-\frac{n}{2t}$

\begin{align}
\nonumber
&\frac{\Delta u}{u}
-
\alpha
\sum_{i,j=1,i\neq j}^{n}
\frac{u_{x_i x_j}(x,t)}{u(x,t)}
-
\beta
\sum_{i,j=1,i\neq j}^{n} 
\frac{u_{x_i}(x,t)}{u(x,t)}\cdot\frac{u_{x_j}(x,t)}{u(x,t)}
-
\gamma
\frac{|\nabla u|^2}{u^2}
&& \text{\tssteelblue{
}} 
\\[1ex] \notag
&\ge
\int_{\mathbb{R}^n}
\left(
-\frac{n}{2t}
+
\left(
1-
(n-1)
\left(
\alpha+\beta
\right)
-\gamma
\right)
\frac{|x-y|^2}{4t^2}
\right)
\frac{\frac{1}{(4\pi t)^{n/2}}e^{-\frac{|x-y|^{2}}{4t}}g(y)}{u(x,t)}
\,dy
&& \text{\tssteelblue{$1-
(n-1)
\left(
\alpha+\beta
\right)
-\gamma
\ge0$
}} 
\\[1ex] \notag
&\ge
\int_{\mathbb{R}^n}
-\frac{n}{2t}
\frac{\frac{1}{(4\pi t)^{n/2}}e^{-\frac{|x-y|^{2}}{4t}}g(y)}{u(x,t)}
\,dy
&& \text{\tssteelblue{$\int_{\mathbb{R}^n}\frac{\frac{1}{(4\pi t)^{n/2}}e^{-\frac{|x-y|^{2}}{4t}}g(y)}{u(x,t)}\,dy=1$
}} 
\\[1ex] \notag
&=
-\frac{n}{2t}
&& \text{\tssteelblue{
}} 
\end{align}

\end{enumerate}

\end{proof}


We generalize the Li-Yau inequality from second derivatives to fourth derivatives.

\begin{theorem}[\deeppink{Li-Yau type inequality for fourth derivatives}]\label{thm: Li-Yau Type Inequality for fourth derivatives n dim}
Let $u=u(x,t)$ be given by \eqref{eqn: representation formula for the heat equation Li-Yau}. Then
\begin{align}
\nonumber
&
\sum_{i=1}^{n}
\frac{u_{x_i x_i x_i x_i}}{u}
+
k_1
\sum_{i,j=1,i\neq j}^{n}
\frac{u_{x_i x_i x_j x_j}}{u}
+
k_2
\left|\frac{\nabla u}{u}\right|^4
+
k_3
\left|\frac{\Delta u}{u}\right|^2
+
k_4
\left(
\sum_{i,j=1,i\neq j}^{n} \frac{u_{x_i x_j}}{u}
\right)^2
&& \text{\tssteelblue{
}}  
\\[1ex] \notag
&\ge
\left(3 n+k_1+n^2k_3 -\frac{n \left(3+(n-1)k_1+nk_3\right){}^2}{1+n\left(k_2+k_3\right)}\right)
\frac{1}{4t^2},
&& \text{\tssteelblue{
}} 
\end{align} 
provided that $k_1$, $k_2$, $k_3$, and $k_4$ are constants satisfying
\begin{subequations}
\begin{eqnarray}
k_2+k_3& > & -\frac{1}{n}, \\[2mm]
k_1 & \ge & -nk_4, \\[2mm]    
k_2 & \le & 0, \\[2mm]
k_3 & \le  & 0, \\[2mm]
k_4 & \le & 0.
\end{eqnarray}
\end{subequations}


\end{theorem}

\begin{proof}
\ \


\begin{enumerate}[(1)]
  \item  
  Estimate on $\sum_{i=1}^{n}\frac{u_{x_i x_i x_i x_i}(x,t)}{u(x,t)}$
\begin{align}
\nonumber
&\sum_{i=1}^{n}
\frac{u_{x_i x_i x_i x_i}(x,t)}{u(x,t)}
&& \text{\tssteelblue{Lemma~\ref{lem: derivatives of representation formula solution of heat equation} (e)
}}  
\\[1ex] \notag
&=
\sum_{i=1}^{n}
\int_{\mathbb{R}^n}
\frac{12 t^2-12 t (x_{i}-y_{i})^2+(x_{i}-y_{i})^4}{16 t^4}
\frac{\frac{1}{(4\pi t)^{n/2}}e^{-\frac{|x-y|^{2}}{4t}}g(y)}{u(x,t)}
\,dy
&& \text{\tssteelblue{
}} 
\\[1ex] \notag
&=
\int_{\mathbb{R}^n}
\sum_{i=1}^{n}
\frac{12 t^2-12 t (x_{i}-y_{i})^2+(x_{i}-y_{i})^4}{16 t^4}
\frac{\frac{1}{(4\pi t)^{n/2}}e^{-\frac{|x-y|^{2}}{4t}}g(y)}{u(x,t)}
\,dy
&& \text{\tssteelblue{$\sum\int=\int\sum$
}}
\\[1ex] \notag
&=
\int_{\mathbb{R}^n}
\sum_{i=1}^{n}
\frac{12 t^2-12 t (x_{i}-y_{i})^2+(x_{i}-y_{i})^4}{16 t^4}
\frac{\frac{1}{(4\pi t)^{n/2}}e^{-\frac{|x-y|^{2}}{4t}}g(y)}{u(x,t)}
\,dy
&& \text{\tssteelblue{$|x-y|^2:=\sum_{i=1}^{n} (x_{i}-y_{i})^2$; use \eqref{eqn: 4^1 4th derivative L-Y ineq}
}}
\\[1ex] \notag
&\ge
\int_{\mathbb{R}^n}
\frac{12 n t^2-12 t |x-y|^2+\frac{1}{n}|x-y|^4}{16 t^4}
\frac{\frac{1}{(4\pi t)^{n/2}}e^{-\frac{|x-y|^{2}}{4t}}g(y)}{u(x,t)}
\,dy
&& \text{\tssteelblue{
}}
\end{align}

\begin{align}
\nonumber
&\int_{\mathbb{R}^n}
\sum_{i=1}^{n}
\frac{(x_{i}-y_{i})^4}{16 t^4}
\frac{\frac{1}{(4\pi t)^{n/2}}e^{-\frac{|x-y|^{2}}{4t}}g(y)}{u(x,t)}
\,dy
&& \text{\tssteelblue{
}}  
\\[1ex] \notag
&=
\int_{\mathbb{R}^n}
\sum_{i=1}^{n}
\left(
\frac{x_{i}-y_{i}}{2t}
\right)^4
\frac{\frac{1}{(4\pi t)^{n/2}}e^{-\frac{|x-y|^{2}}{4t}}g(y)}{u(x,t)}
\,dy
&& \text{\tssteelblue{$\sum_{i=1}^{n} a_i^2\ge \frac{1}{n}\left(\sum_{i=1}^{n} a_i\right)^2$
}}
\\[1ex] \notag
&\ge
\int_{\mathbb{R}^n}
\frac{1}{n}
\left(
\sum_{i=1}^{n}
\left(
\frac{x_{i}-y_{i}}{2t}
\right)^2
\right)^2
\frac{\frac{1}{(4\pi t)^{n/2}}e^{-\frac{|x-y|^{2}}{4t}}g(y)}{u(x,t)}
\,dy
&& \text{\tssteelblue{$|x-y|^2:=\sum_{i=1}^{n} (x_{i}-y_{i})^2$
}}
\\[1ex] \notag
&=
\int_{\mathbb{R}^n}
\frac{1}{n}
\left(
\frac{|x-y|^2}{4t^2}
\right)^2
\frac{\frac{1}{(4\pi t)^{n/2}}e^{-\frac{|x-y|^{2}}{4t}}g(y)}{u(x,t)}
\,dy
&& \text{\tssteelblue{expand
}} 
\\[1ex] \label{eqn: 4^1 4th derivative L-Y ineq}
&=
\int_{\mathbb{R}^n}
\frac{1}{n}
\frac{|x-y|^4}{16t^4}
\frac{\frac{1}{(4\pi t)^{n/2}}e^{-\frac{|x-y|^{2}}{4t}}g(y)}{u(x,t)}
\,dy
&& \text{\tssteelblue{
}} 
\end{align}

  \item  
  Estimate on $\sum_{i,j=1,i\neq j}^{n}\frac{u_{x_i x_i x_j x_j}(x,t)}{u(x,t)}$
  
\begin{align}
\nonumber
&\sum_{i,j=1,i\neq j}^{n}
\frac{u_{x_i x_i x_j x_j}(x,t)}{u(x,t)}
&& \text{\tssteelblue{Lemma~\ref{lem: derivatives of representation formula solution of heat equation} (f)
}} 
\\[1ex] \notag
&=
\sum_{i,j=1,i\neq j}^{n}
\int_{\mathbb{R}^n}
\frac{4t^2-2t(x_{i}-y_{i})^2-2t(x_{j}-y_{j})^2+(x_{i}-y_{i})^2(x_{j}-y_{j})^2}{16t^4}
\frac{\frac{1}{(4\pi t)^{n/2}}e^{-\frac{|x-y|^{2}}{4t}}g(y)}{u(x,t)}
\,dy
&& \text{\tssteelblue{$\sum\int=\int\sum$
}} 
\\[1ex] \notag
&=
\int_{\mathbb{R}^n}
\sum_{i,j=1,i\neq j}^{n}
\frac{4t^2-2t(x_{i}-y_{i})^2-2t(x_{j}-y_{j})^2+(x_{i}-y_{i})^2(x_{j}-y_{j})^2}{16t^4}
\frac{\frac{1}{(4\pi t)^{n/2}}e^{-\frac{|x-y|^{2}}{4t}}g(y)}{u(x,t)}
\,dy
&& \text{\tssteelblue{$|x-y|^2:=\sum_{i=1}^{n} (x_{i}-y_{i})^2$
}} 
\\[1ex] \notag
&\deeppink{=}
\int_{\mathbb{R}^n}
\frac{4t^2
-\deeppink{4(n-1)}t|x-y|^2
+
\sum_{i,j=1,i\neq j}^{n}(x_{i}-y_{i})^2(x_{j}-y_{j})^2}{16t^4}
\frac{\frac{1}{(4\pi t)^{n/2}}e^{-\frac{|x-y|^{2}}{4t}}g(y)}{u(x,t)}
\,dy,
&& \text{\tssteelblue{
}} 
\end{align}   
where we have used \eqref{eqn: sum (xi-yi)^2+(xi-yi)^2=2(n-1)|x-y|^2}.  
  \item 
  Estimate on $\left|\frac{\nabla u(x,t)}{u(x,t)}\right|^4$

\begin{align}
\nonumber
&\left|\frac{\nabla u(x,t)}{u(x,t)}\right|^4
:=
\left(
\sum_{i=1}^{n} \left|\frac{u_{x_i}(x,t)}{u(x,t)}\right|^2
\right)^2
&& \text{\tssteelblue{definition of $\nabla$
}} 
\\[1ex] \notag
&
=
\left(
\sum_{i=1}^{n}
\left|
\int_{\mathbb{R}^n}-\frac{x_{i}-y_{i}}{2t}
\frac{\frac{1}{(4\pi t)^{n/2}}e^{-\frac{|x-y|^{2}}{4t}}g(y)}{u(x,t)}
\,dy
\right|^2
\right)^2
&& \text{\tssteelblue{Lemma~\ref{lem: derivatives of representation formula solution of heat equation} (a)
}} 
\\[1ex] \notag
&\le
\left(
\sum_{i=1}^{n}
\int_{\mathbb{R}^n}
\left(\frac{x_{i}-y_{i}}{2t}\right)^2
\frac{\frac{1}{(4\pi t)^{n/2}}e^{-\frac{|x-y|^{2}}{4t}}g(y)}{u(x,t)}
\,dy
\right)^2
&& \text{\tssteelblue{$\int_{\mathbb{R}^n}\frac{\frac{1}{(4\pi t)^{n/2}}e^{-\frac{|x-y|^{2}}{4t}}g(y)}{u(x,t)}\,dy=1$ $\implies$ Jensen ineq.
}}
\\[1ex] \notag
&=
\left(
\int_{\mathbb{R}^n}
\sum_{i=1}^{n}
\left(\frac{x_{i}-y_{i}}{2t}\right)^2
\frac{\frac{1}{(4\pi t)^{n/2}}e^{-\frac{|x-y|^{2}}{4t}}g(y)}{u(x,t)}
\,dy
\right)^2
&& \text{\tssteelblue{$\int_{\mathbb{R}^n}\sum_{i=1}^{n}=\sum_{i=1}^{n}\int_{\mathbb{R}^n}$; $|x-y|^2:=\sum_{i=1}^{n} (x_{i}-y_{i})^2$
}} 
\\[1ex] \notag
&=
\left(
\int_{\mathbb{R}^n}
\frac{|x-y|^2}{4t^2}
\frac{\frac{1}{(4\pi t)^{n/2}}e^{-\frac{|x-y|^{2}}{4t}}g(y)}{u(x,t)}
\,dy
\right)^2
&& \text{\tssteelblue{$\int_{\mathbb{R}^n}\frac{\frac{1}{(4\pi t)^{n/2}}e^{-\frac{|x-y|^{2}}{4t}}g(y)}{u(x,t)}\,dy=1$ $\implies$ Jensen ineq.
}} 
\\[1ex] \notag
&\le
\int_{\mathbb{R}^n}
\frac{\left|x-y\right|^4}{16t^4}
\frac{\frac{1}{(4\pi t)^{n/2}}e^{-\frac{|x-y|^{2}}{4t}}g(y)}{u(x,t)}
\,dy
&& \text{\tssteelblue{
}} 
\end{align}  
  
  \item 
  Estimate on $\left|\frac{\Delta u(x,t)}{u(x,t)}\right|^2$
  
\begin{align}
\nonumber
&\left|\frac{\Delta u(x,t)}{u(x,t)}\right|^2
:=
\left(
\sum_{i=1}^{n} \frac{u_{x_i x_i}(x,t)}{u(x,t)}
\right)^2
&& \text{\tssteelblue{definition of $\Delta$
}} 
\\[1ex] \notag
&
=
\left(
\sum_{i=1}^{n}
\int_{\mathbb{R}^n}
\left(
-\frac{1}{2t}
+
\frac{(x_{i}-y_{i})^2}{4t^2}
\right)
\frac{\frac{1}{(4\pi t)^{n/2}}e^{-\frac{|x-y|^{2}}{4t}}g(y)}{u(x,t)}
\,dy
\right)^2
&& \text{\tssteelblue{Lemma~\ref{lem: derivatives of representation formula solution of heat equation} (b)
}} 
\\[1ex] \notag
&=
\left(
\int_{\mathbb{R}^n}
\sum_{i=1}^{n}
\left(
-\frac{1}{2t}
+
\frac{(x_{i}-y_{i})^2}{4t^2}
\right)
\frac{\frac{1}{(4\pi t)^{n/2}}e^{-\frac{|x-y|^{2}}{4t}}g(y)}{u(x,t)}
\,dy
\right)^2
&& \text{\tssteelblue{$\int_{\mathbb{R}^n}\sum_{i=1}^{n}=\sum_{i=1}^{n}\int_{\mathbb{R}^n}$; $|x-y|^2:=\sum_{i=1}^{n} (x_{i}-y_{i})^2$
}} 
\\[1ex] \notag
&=
\left(
\int_{\mathbb{R}^n}
\left(
-\frac{n}{2t}
+
\frac{|x-y|^2}{4t^2}
\right)
\frac{\frac{1}{(4\pi t)^{n/2}}e^{-\frac{|x-y|^{2}}{4t}}g(y)}{u(x,t)}
\,dy
\right)^2
&& \text{\tssteelblue{$\int_{\mathbb{R}^n}\frac{\frac{1}{(4\pi t)^{n/2}}e^{-\frac{|x-y|^{2}}{4t}}g(y)}{u(x,t)}\,dy=1$ $\implies$ Jensen ineq.
}} 
\\[1ex] \notag
&\le
\int_{\mathbb{R}^n}
\frac{\left(|x-y|^2-2nt\right)^2}{16t^4}
\frac{\frac{1}{(4\pi t)^{n/2}}e^{-\frac{|x-y|^{2}}{4t}}g(y)}{u(x,t)}
\,dy
&& \text{\tssteelblue{
}} 
\end{align}  
  
  \item 
  Estimate on $\left(
\sum_{i,j=1,i\neq j}^{n} \frac{u_{x_i x_j}(x,t)}{u(x,t)}
\right)^2$
\begin{align}
\nonumber
&
\sum_{i,j=1,i\neq j}^{n}
\left( 
\frac{u_{x_i x_j}(x,t)}{u(x,t)}
\right)^2
&& \text{\tssteelblue{
}} 
\\[1ex] \notag
&
=
\sum_{i,j=1,i\neq j}^{n}
\left(
\int_{\mathbb{R}^n}
\frac{(x_{i}-y_{i})(x_{j}-y_{j})}{4t^2}
\frac{\frac{1}{(4\pi t)^{n/2}}e^{-\frac{|x-y|^{2}}{4t}}g(y)}{u(x,t)}
\,dy
\right)^2
&& \text{\tssteelblue{Lemma~\ref{lem: derivatives of representation formula solution of heat equation} (c)
}} 
\\[1ex] \notag
&\le
\left(
\sum_{i,j=1,i\neq j}^{n}
\int_{\mathbb{R}^n}
\frac{|x_{i}-y_{i}||x_{j}-y_{j}|}{4t^2}
\frac{\frac{1}{(4\pi t)^{n/2}}e^{-\frac{|x-y|^{2}}{4t}}g(y)}{u(x,t)}
\,dy
\right)^2
&& \text{\tssteelblue{$\sum_{i}  a_i^2
=
\sum_{i} |a_i|^2\le
\left(\sum_{i} |a_i|
\right)^2$   
}} 
\\[1ex] \notag
&=
\left(
\int_{\mathbb{R}^n}
\sum_{i,j=1,i\neq j}^{n}
\frac{|x_{i}-y_{i}||x_{j}-y_{j}|}{4t^2}
\frac{\frac{1}{(4\pi t)^{n/2}}e^{-\frac{|x-y|^{2}}{4t}}g(y)}{u(x,t)}
\,dy
\right)^2
&& \text{\tssteelblue{$\int_{\mathbb{R}^n}\sum_{i,j=1,i\neq j}^{n}=\sum_{i,j=1,i\neq j}^{n}\int_{\mathbb{R}^n}$
}}
\\[1ex] \notag
&\le
\int_{\mathbb{R}^n}
\left(
\sum_{i,j=1,i\neq j}^{n}
\frac{|x_{i}-y_{i}||x_{j}-y_{j}|}{4t^2}
\right)^2
\frac{\frac{1}{(4\pi t)^{n/2}}e^{-\frac{|x-y|^{2}}{4t}}g(y)}{u(x,t)}
\,dy
&& \text{\tssteelblue{$\int_{\mathbb{R}^n}\frac{\frac{1}{(4\pi t)^{n/2}}e^{-\frac{|x-y|^{2}}{4t}}g(y)}{u(x,t)}\,dy=1$ $\implies$ Jensen ineq.
}}
\\[1ex] \notag
&\le
\int_{\mathbb{R}^n}
\deeppink{n}
\sum_{i,j=1,i\neq j}^{n}
\frac{(x_{i}-y_{i})^2(x_{j}-y_{j})^2}{16t^4}
\frac{\frac{1}{(4\pi t)^{n/2}}e^{-\frac{|x-y|^{2}}{4t}}g(y)}{u(x,t)}
\,dy
&& \text{\tssteelblue{$\left(\sum_{i=1}^{n} a_i\right)^2 \leq n\sum_{i=1}^{n} a_i^2$}
}
\end{align}

\end{enumerate}

Combining the above calculations, we have

\begin{align}
\nonumber
&
\sum_{i=1}^{n}
\frac{u_{x_i x_i x_i x_i}}{u}
+
k_1
\sum_{i,j=1,i\neq j}^{n}
\frac{u_{x_i x_i x_j x_j}}{u}
+
k_2
\left|\frac{\nabla u}{u}\right|^4
+
k_3
\left|\frac{\Delta u}{u}\right|^2
+
k_4
\left(
\sum_{i,j=1,i\neq j}^{n} \frac{u_{x_i x_j}}{u}
\right)^2
&& \text{\tssteelblue{
}}  
\\[1ex] \notag
&\ge
\int_{\mathbb{R}^n}
\frac{h(t,x,y,n)}{16 t^4}
\frac{\frac{1}{(4\pi t)^{n/2}}e^{-\frac{|x-y|^{2}}{4t}}g(y)}{u(x,t)}
\,dy,
&& \text{\tssteelblue{
}} 
\end{align}  
where 
\begin{align}
\nonumber
h(t,x,y,n)
&=12 n t^2-12 t |x-y|^2+\frac{1}{n}|x-y|^4
+
k_1
\left(
\darkmagenta{4t^2}
-\deeppink{4(n-1)}t|x-y|^2
+
\sum_{i,j=1,i\neq j}^{n}(x_{i}-y_{i})^2(x_{j}-y_{j})^2
\right)
&& \text{\tssteelblue{
}}  
\\[1ex] \notag
&\hspace{4mm}+
k_2
\left|x-y\right|^4
+
k_3
\left(|x-y|^2-2nt\right)^2
+
k_4
\left(
n\sum_{i,j=1,i\neq j}^{n}(x_{i}-y_{i})^2(x_{j}-y_{j})^2
\right)
&& \text{\tssteelblue{
}} 
\\[1ex] \notag
&=
\left(
12n+4k_1+4n^2k_3
\right)t^2
+
\left(
-12-\deeppink{4(n-1)}k_1-4nk_3
\right)
t\left|x-y\right|^2
+
\left(
\frac{1}{n}+k_2+k_3
\right)|x-y|^4
&& \text{\tssteelblue{
}}
\\[1ex] \notag
&\hspace{4mm}+
\left(
k_1+nk_4
\right)
\sum_{i,j=1,i\neq j}^{n}(x_{i}-y_{i})^2(x_{j}-y_{j})^2
&& \text{\tssteelblue{
}}
\\[1ex] \notag
&=Ct^2+Bt|x-y|^2+A|x-y|^4
+
\left(
k_1+nk_4
\right)
\sum_{i,j=1,i\neq j}^{n}(x_{i}-y_{i})^2(x_{j}-y_{j})^2
&& \text{\tssteelblue{
}}
\\[1ex] \notag
&=
A\left(|x-y|^2+\frac{B}{2A}t\right)^2
+
\left(\frac{4AC-B^2}{4A}\right)t^2
+
\left(
k_1+nk_4
\right)
\sum_{i,j=1,i\neq j}^{n}(x_{i}-y_{i})^2(x_{j}-y_{j})^2,
&& \text{\tssteelblue{
}}
\end{align}  
where 
\begin{subequations}
\begin{eqnarray}
A &=& \frac{1}{n}+k_2+k_3, \\[2mm]
B &=& -12-\deeppink{4(n-1)}k_1-4nk_3, \\[2mm]    
C &=& 12n+4k_1+4n^2k_3, \\[2mm]
\frac{4AC-B^2}{4A} &=& 4 \left(3 n+k_1+n^2k_3 -\frac{n \left(3+(n-1)k_1+nk_3\right){}^2}{1+n\left(k_2+k_3\right)}\right).
\end{eqnarray}
\end{subequations}

\end{proof}

\bibliographystyle{amsplain}

\end{document}